\documentclass[reqno]{amsart}

\usepackage{verbatim} 
\usepackage{amsmath}
\usepackage{amsfonts}
\usepackage{amssymb}
\usepackage{mathrsfs}
\usepackage{amsthm}
\usepackage{newlfont}

\usepackage{fullpage,amsmath, amssymb,amscd}
\usepackage{latexsym, epsfig, color}
\usepackage[mathscr]{eucal}
\usepackage{amscd}
\usepackage[all, cmtip]{xy}

\renewcommand{\thefootnote}{\fnsymbol{footnote}}
\newcommand{\authorfootnotes}{\renewcommand\thefootnote{\@fnsymbol\c@footnote}}%
 
\theoremstyle{plain} % theorem-like things
\newtheorem{theorem}{Theorem}
\newtheorem{lemma}[theorem]{Lemma}

\newtheorem{corollary}[theorem]{Corollary}
\newtheorem{proposition}[theorem]{Proposition}

\theoremstyle{definition} % definition-like things

\theoremstyle{remark} % remark-like things

\numberwithin{theorem}{section}
\numberwithin{equation}{section}

\newcommand{\texteqn}[1]{\hspace{4pt}\text{#1}\hspace{4pt}}

\def\e{\varepsilon}

\newcommand{\ga}{\alpha}

\newcommand{\gd}{\delta}
\newcommand{\gD}{\Delta}
\newcommand{\gs}{\sigma}

\newcommand{\om}{\omega}

\newcommand{\mf}[1]{\mathfrak{#1}}
\newcommand{\mc}[1]{\mathcal{#1}}

\newcommand{\ot}[1]{\widetilde{#1}}
\newcommand{\ol}[1]{\overline{#1}}
\newcommand{\ul}[1]{\underline{#1}}
\newcommand{\upone}[1]{{#1}^{(1)}}

\def\Z{\mathbb Z}

\def\Q{\mathbb Q}
\def\R{\mathbb R}

\def\F{\mathbb F}

\def\deg{\operatorname{deg}}
\def\det{\operatorname{det}}
\def\tr{\operatorname{tr}}

\def\End{\operatorname{End}}

\def\Gal{\operatorname{Gal}}

\def\Centr{\operatorname{Centr}}
\def\Frob{\operatorname{Frob}}

\def\mod{\operatorname{mod}}

\def\lcm{\operatorname{lcm}}

\def\GL{\operatorname{GL}}

\def\log{\operatorname{log}}

\def\Tr{\operatorname{Tr}}

\begin{document}
\title[]{The normal order of of the divisor-counting function for invariants of rank 2 Drinfeld modules}
\author{
Abel Castillo
}

\address{
\begin{itemize}
\item[-]
Department of Mathematics, Statistics and Computer Science, University of Illinois at Chicago, 851 S Morgan St, 322
SEO, Chicago, 60607, IL, USA;
\end{itemize}
} 
\email{abelcastillomath@gmail.com}

\date{May 16, 2015}

\begin{abstract} 
We compute  the first and second moments of the divisor-counting function for the Euler-Poincar\'{e} characteristic and the trace of Frobenius for the reductions modulo $p$ of a rank 2 Drinfeld module with nontrivial endomorphism ring, as the prime $p$ varies over the primes of ordinary reduction of the Drinfeld module. From these moments we derive the normal order of the number of prime divisors of these invariants.
\end{abstract}
\maketitle

\section{Introduction}
In 1917, Hardy and Ramanujan \cite{HardyRamanujan1917} showed that the normal order of the number of prime divisors of $n$ is $\log \log n$. More precisely, write $\nu (n)$ for the number of prime divisors of $n$, and for any $\e > 0$ they showed that
$$
\#\{n \in \Z, 1 \leq n \leq x: \left| \nu(n) - \log \log n \right| > \e \log \log n  \} = \ul o(x).
$$

In 1934, Paul Tur\`an \cite{Turan1934} reproved this fact as an immediate consequence of the second central moment inequality
$$
\sum_{1 \leq n \leq x}\left( \nu(n) - \log \log n \right)^2 \ll n \log \log n.
$$

Later, Paul Erd\"os \cite{Erdos1935Normal} showed that the normal order of $\nu(p-1)$ as $p$ runs through primes is $\log \log p$. Erd\"os did not use the technique of Tur\`an; nonetheless, this result can be derived using the technique of Tur\`an combined with information about primes in arithmetic progressions, in the form of the celebrated Bombieri-Vinogradov theorem. This result of Erd\"os tells us that, for a ``typical'' prime $p$, the number of prime divisors of $p-1$ is the same as the number of prime divisors of a ``typical'' natural number of roughly the same size.

It is natural to explore function field analogues of the questions described above, where we replace $\Z$ with $A:=\F_q[T]$ and $\Q$ with $F:=\F_q(T)$. Since $A$ is a unique factorization domain, nonzero prime ideals correspond to monic irreducible polynomials in $T$, and we define $\nu(\cdot)$ to be the function on $A$ which counts prime divisors, that is, for $n \in A$, $
\nu(n) = \# \{p \in A \texteqn{monic irreducible: } p \mid n \}.$
Yu-Ru Liu proves in \cite{Liu2004} that the normal order of $\nu(n)$ is $\log \deg n$. One can also show that the normal order of $\nu(p-1)$, where $p$ runs through monic irreducible polynomials, is $\log \deg p$. To do this, one can take the techique of Tur\`an and combine it with the fact that, for $n$ in $A - \F_q$, we have \cite[Theorem 4.8]{Rosen2002}:
$$
\# \{p \in A \texteqn{monic irreducible: } \deg p = x, n \mid p-1  \}
=
\frac{q^x}{x}
+ O \left(q^{x/2} \deg n \right).
$$

We observe that the condition ``$n \mid p-1 $'' implies that $p$ splits completely in the $n$-division field of the Carlitz module \cite[Theorem 12.10]{Rosen2002}. The Carlitz module is a rank 1 Drinfeld $A$-module over $F$, whose $n$-division fields have Galois groups isomorphic to $(A/nA)^*$ for $n$ in $A - \F_q$. This suggests looking for analogues of this result associated to Drinfeld modules of higher rank.

Let $\tau$ be the map $x \mapsto x^q$, and let $F\{\tau\}$ be the twisted polynomial ring where multiplication satisfies $\tau a = a^q \tau$ for all $a \in F$. A Drinfeld $A$-module over $F$ of rank $r$ is given by an $\F_q$-algebra homomorphism $\phi:A \to F\{ \tau \}$, where the image of $T$ is 
$$
\phi_T = T + c_1 \tau + \cdots + c_r \tau^r,
$$ 
with $c_i \in F$, $c_r \neq 0$, and $r \geq 1$. This homomorphism endows $F$ with a module structure different from that coming from the natural inclusion of $A$ in $F$. When referring to $F$ with the module structure coming from $\phi$, one writes $\phi(F)$. The Carlitz module $C$ mentioned above is the Drinfeld $A$-module over $F$ with $C_T = T + \tau$.

For a prime $p$ of $A$ of ``good reduction'' for $\phi$, one can talk about the reduction of $\phi$ modulo $p$. This is a Drinfeld $A$-module over the residue field $\F_{\mf p} := A / \mf p$ (with $\mf p = pA$), and we obtain a new $A$-module structure on $\F_{\mf p}$. The Euler-Poincar\'e characteristic of $\phi(\F_{\mf p})$ is an ideal in $A$; we will write $\varpi_p$ for a generator of this ideal, and we can ask questions about the number of prime divisors of $\varpi_p$ as $p$ varies.

For $n$ in $A - \F_q$, the $n$-division field of $\phi$ over $F$ is the extension of $F$ obtained from adjoining $n$-torsion points of $\phi( \ol F)$, the algebraic closure of $F$ endowed with the module structure coming from $\phi$. This extension is Galois over $F$, and there is an injection of its Galois group into the general linear group $GL_r(A/nA)$ that is well-defined up to conjugation. Using this homomorphism, the Artin symbol at a prime $p$ of $A$ determines a conjugacy-invariant subset of $GL_r(A/nA)$, and one can talk about the trace of these matrices. As $n$ varies, one gets homomorphisms that are ``compatible'' in the sense that there exists a unique element $a_p$ in $A$ which reduces modulo $n$ to the trace of the image of the Artin symbol under these homomorphisms. We can therefore ask about the number of prime divisors of this element as $p$ varies.

The ring $\End_{\ol F}(\phi)$ is the subring of $\ol F\{\tau\}$ of elements which commute with the image of $\phi$. It necessarily contains a subring isomorphic to $A$; when $\End_{\ol F}(\phi) \cong A$, we say that $\phi$ has trivial endomorphism ring. Otherwise, it is known that $\End_{\ol F}(\phi)$ has $A$-rank of at most $r$, and is isomorphic to an order in a field extension of $F$.

In \cite{Cojocaru2008Acta}, A. C. Cojocaru computes the limiting distribution for the number of prime divisors of of $\varpi_p$ and $a_p$ for Drinfeld $A$-modules of rank $r \geq 2$ with trivial endomorphism ring as $p$ varies. For both quantities, the results of Cojocaru can be interpreted to say, for instance, that for a ``typical'' prime $p$ of $A$, $\varpi_p$ has the same number of prime divisors as a ``typical'' element of $A$ of the same degree as $p$.

Prior to \cite{Cojocaru2008Acta}, one finds analogues of the results of \cite{Erdos1935Normal} in the number field setting, in the study of the number of prime divisors of the quantities $ \#E_p(\F_p)$ (the number of $\F_p$-rational points of the reduction modulo $p$ of an elliptic curve $E$) and $a_p := p + 1 - \#E_p(\F_p)$ (the trace of Frobenius of an elliptic curve $E$ at $p$) as $p$ varies. K. Murty and R. Murty \cite{MurtyMurty1984} explore the sequence of Fourier coefficients $a_p$ of an eigenform as $p$ varies; A. Miri and K. Murty \cite{MiriMurty2001} and Cojocaru \cite{Cojocaru2005Acta} explore the values of $ \#E_p(\F_p)$ as $p$ varies.

In this paper, we extend the results of \cite{Cojocaru2008Acta} to Drinfeld $A$-modules of rank $r=2$ with nontrivial endomorphism ring when $q$ is an odd prime power. In this setting, the endomorphism ring is necessarily an order in a quadratic extension of $F$. We obtain results by studying the Galois groups of $n$-division fields of such Drinfeld modules and combining the information coming from suitable $n$ to calculcate the first and second moments of the divisor-counting function. In particular, we are able to identify these Galois groups with quotients of the endomorphism ring in such a way that the conditions ``$n \mid \varpi_p$'' and ``$n \mid a_p$'' reduce to conditions whose solution sets are straightforward to count.

\subsection{Notation}
Let $q$ be a power of an odd rational prime, and write $\F_q$ for the finite field with $q$ elements. Let $A:=\F_q[T]$ and write $F := \F_q(T)$ for its field of fractions. Define $|\cdot|$ to be the valuation at infinity of $F$, which can also be constructed by taking $|0| = 0 $, $|n| = q^{\deg n} $ for all nonzero $n$ in $A$, and extending multiplicatively to $F$. This use of $|\cdot|$ should be distinguishible from its use as absolute value of a real number or cardinality of a finite set from context.

Let $\upone A$ be the subset of elements of $A$ that are monic in $T$. The prime ideals of $A$ correspond to irreducible elements in $\upone A$, and we will call these ``primes'' and ``monic irreducibles'' interchangeably. For $\mf p = pA$ a prime ideal, write $\F_{\mf p}$ for the residue field $A / \mf p$. For a field $L$, write $\ol L$ for its separable closure.

The statements $f(x) = O(g(x))$ and $f(x) \ll g(x)$ are both taken to imply that there exists a constant $c \in \R$ such that, for all $x \in [0,\infty)$, we have $|f(x)| \leq c |g(x)|$. Subscripts in either symbol are meant to indicate objects on which the constant $c$ depends; for instance, the statement $f \ll_\phi g$ means that, for all $x \in [0,\infty)$, $|f(x)| \leq c |g(x)|$ for a constant $c$ that depends only on $\phi$.

\subsection{Statement of Main Results}

Our main result is a computation of the first and second moments of $\nu(a_p)$ as $p$ runs through primes of ordinary reduction for a fixed rank $2$ Drinfeld module $\phi$ with nontrivial endomorphism ring.

\begin{theorem}\label{thm:MomentsNTER}
Let $\phi$ be a Drinfeld module over $F$ with nontrivial endomorphism ring. Let $\End_{\ol F}(\phi)$, the ring of endomorphisms of $\phi$, be isomorphic to an order in an imaginary quadratic extension $E$ of $F$. Let $\F_{\phi}$ be the largest constant field contained in any $n$-division field of $\phi$ over $F$. 
\begin{itemize}
\item[I.] Assume that $E \neq \F_{q^2}(T)$. If $k \in \{1,2\}$ and $x$ is a sufficiently large positive integer,
$$
\sum_{\substack{
 \deg p = x \\ 
 p \; \mathrm{good}\;\mathrm{reduction}\;\mathrm{for}\;\phi \\
 p \; \mathrm{splits}\;\mathrm{in}\; E
}} 
\nu(a_p)^k 
= \frac{1}{2}\frac{q^x}{x} (\log x)^k 
+ O \left( \frac{q^x}{x}(\log x)^{k-1} \right)
$$
\item[]and
$$
\sum_{\substack{
 \deg p = x \\ 
 p \; \mathrm{good}\;\mathrm{reduction}\;\mathrm{for}\;\phi \\
 p \; \mathrm{splits}\;\mathrm{in}\; E
}} 
\nu(\varpi_p)^k 
= \frac{1}{2}\frac{q^x}{x} (\log x)^k 
+ O \left( \frac{q^x}{x}(\log x)^{k-1} \right).
$$
\item[II.] Assume that $E = \F_{q^2}(T)$. If $k \in \{1,2\}$ and $x$ is a sufficiently large positive even integer,
$$
\sum_{\substack{
 \deg p \leq x \\ 
 p \; \mathrm{good}\;\mathrm{reduction}\;\mathrm{for}\;\phi \\
 p \; \mathrm{splits}\;\mathrm{in}\; E
}}
 \nu(a_p)^k 
= \frac{q^x}{x}(\log x)^k  
+ O \left( \frac{q^x}{x}(\log x)^{k-1} \right)
$$
\item[]and
$$
\sum_{\substack{
 \deg p \leq x \\ 
 p \; \mathrm{good}\;\mathrm{reduction}\;\mathrm{for}\;\phi \\
 p \; \mathrm{splits}\;\mathrm{in}\; E
}}
\nu(\varpi_p)^k
= \frac{q^x}{x} (\log x)^k 
+ O \left( \frac{q^x}{x}(\log x)^{k-1} \right).
$$
\item[](These sums are vacuously $0$ if $x$ is odd.)
\end{itemize}
\end{theorem}

Using Theorem \ref{thm:MomentsNTER}, we establish the following facts; in particular, Corollary \ref{thm:NormalOrderNTER}  states that the normal order of $\nu(a_p)$ and $\nu(\varpi_p)$ is $\log \deg p$ (as $p$ runs through primes where $a_p$ can be nonzero). These corollaries follow from standard techniques; for the sake of completeness we include their proofs in the last section.

\begin{corollary}\label{thm:CentralMomentsNTER}
Take the notation and setting of Theorem \ref{thm:MomentsNTER}. Then,
$$
\sum_{\substack{
 \deg p \leq x \\ 
 p \; \mathrm{good}\;\mathrm{reduction}\;\mathrm{for}\;\phi \\
 p \; \mathrm{splits}\;\mathrm{in}\; E
}}
  \left(
\nu(a_p) - \log x
  \right)^2
  \ll_{\phi}
\frac{q^x}{x} \left( \log x \right)
$$
and 
$$
\sum_{\substack{
 \deg p \leq x \\ 
 p \; \mathrm{good}\;\mathrm{reduction}\;\mathrm{for}\;\phi \\
 p \; \mathrm{splits}\;\mathrm{in}\; E
}}
  \left(
\nu(\varpi_p) - \log x
  \right)^2
  \ll_{\phi}
\frac{q^x}{x} \left( \log x \right).
$$
\end{corollary}

\begin{corollary}\label{thm:NormalOrderNTER}
Take the notation and setting of Theorem \ref{thm:MomentsNTER}. Then, for any $\e > 0$,
$$
\#
\left\{
  p, \deg p \leq x,  p \; \mathrm{splits}\;\mathrm{in}\; E,
  \left|
    \nu(a_p) - (\log x)
  \right|
  >
  (\log x)^{\frac{1}{2}} (\log \log x)^{\e}
\right\}
=
\ul o \left( \frac{q^x}{x} \right)
$$
and
$$
\#
\left\{
  p, \deg p \leq x,  p \; \mathrm{splits}\;\mathrm{in}\; E,
  \left|
    \nu(\varpi_p) - (\log x)
  \right|
  >
  (\log x)^{\frac{1}{2}} (\log \log x)^{\e}
\right\}
=
\ul o \left( \frac{q^x}{x} \right)
$$
\end{corollary}

\section{Background}

%%%%%%%%%%%%%%%%%%%%%%%%%%%%%%%%%%%%%
%%%% BEGIN SECTION %%%%%%%%%%%%%%%%%%
%%%%%%%%%%%%%%%%%%%%%%%%%%%%%%%%%%%%%
\subsection{Defining Drinfeld modules}
Let $K$ be an $A$-field, and write $K \{ \tau \}$ for the non-commutative polynomial ring with commutation rule
$$
\tau \ga = \ga^q \tau \; \text{for all}\; \ga \in K.
$$

One can view $K \{ \tau \}$  as a subring of the endomorphism ring of the additive group scheme of $K$. One can also identify elements in $K \{ \tau \}$ with elements in $K[X]$ by identifying $\tau$ with $X^q$; in this case, multiplication in $K \{ \tau \}$ can be identified with composition of polynomials in $K[X]$.

A Drinfeld $A$-module $\phi$ over $K$ is determined by an $A$-module homomorphism
$$
\phi:A \to K\{\tau\}, \; n \mapsto \phi_n, 
$$
where 
\begin{equation} \label{eqn:PhiAtT}
\phi_T := T \tau^0 + c_1 \tau + \cdots + c_r \tau^r, \; c_i \in K, c_r \neq 0.
\end{equation}
By $T$ here we mean the image of $T$ in the $A$-field structure map of $K$. We require $r \geq 1$ (nontriviality) and call $r$ the rank of $\phi$. It follows from the above definition that, in order to specify a Drinfeld module $\phi$, it suffices to specify its value at $T$, i.e. $\phi$ is completely determined by $\phi_T$.

One can define Drinfeld modules more generally: let $E$ be a finite extension of $F$ where the prime at infinity does not split, and let $\mc O$ be an order in $E$. Then, a Drinfeld $\mc O$-module over $K$ is characterized by an $\mc O$-module homomorphism
$$
\mc O \to K\{\tau\}, \; n \mapsto \phi_n,
$$
satisfying a nontriviality condition and a condition comparable to (\ref{eqn:PhiAtT}). 
The theory of Drinfeld $\mc O$-modules strongly resembles that of Drinfeld $A$-modules (see for instance \cite[Chapter 13]{Rosen2002} and \cite[Chapter 4]{Goss1996}). In this paper we will use the the phrase ``Drinfeld module'' to mean ``Drinfeld $A$-module''. 

We will write $\phi(K)$ when we wish to view $K$ as an $A$-module with the module structure given by $n \cdot_\phi \ga = \psi_n(\ga)$ for $n \in A$ and $\ga \in K$.

%%%%%%%%%%%%%%%%%%%%%%%%%%%%%%%%%%%%%
%%%% BEGIN SECTION %%%%%%%%%%%%%%%%%%
%%%%%%%%%%%%%%%%%%%%%%%%%%%%%%%%%%%%%
\subsection{Isogenies and endomorphisms} \label{sec:Endomorphism}
Let $K$ be an $A$-field, and let $\psi$ and $\phi$ be Drinfeld modules over $K$. An isogeny from $\psi$ to $\phi$ is a non-zero element $f$ in $\ol K \{\tau\}$ satisfying $f \psi_n = \phi_n f$ for all $n \in A$. For any subfield $L$ of $\ol K$ containing all of the coefficients of $f$ (viewed as a polynomial in $\tau$), we say that $f$ is defined over $L$. 

Taking $\psi = \phi$ above, we say that such an $f$ is an endomorphism of $\phi$. The $A$-module of all endomorphisms of $\phi$ defined over $L$ is in fact a subring of $L\{\tau\}$, which we denote by $\End_L(\phi)$. We record some properties of this ring below.

\begin{theorem}[{\cite[Section 4.7, p. 79-85]{Goss1996}}] \label{fact:EndRingImaginary}
Let $\phi$ be a Drinfeld module over $K$ of rank $r$. $\End_{\ol K}(\phi)$ is a free $A$-module of $A$-rank at most $r$, and it is an order in a finite-dimensional division algebra over $F$. There exists a finite extension $L$ of $K$ such that $\End_{\ol K}(\phi) = \End_{L}(\phi)$; call the smallest such extension of $K$ the field of definition of $\End_{\ol K}(\phi)$.

If $K$ has generic $A$-characteristic (i.e. the $A$-field structure homomorphism $\iota:A \to K$ has zero kernel), $\End_{\ol K}(\phi)$ is an order in a field extension $E := \End_{\ol K}(\phi) \otimes_A F$ over $F$, where $[E:F] \leq r$ and there is exactly one prime in $E$ lying above the prime at infinity of $F$ (i.e. $E$ is a imaginary extension of $F$). 
\end{theorem}

%%%%%%%%%%%%%%%%%%%%%%%%%%%%%%%%%%%%%
%%%% BEGIN SECTION %%%%%%%%%%%%%%%%%%
%%%%%%%%%%%%%%%%%%%%%%%%%%%%%%%%%%%%%
\subsection{The reduction modulo $\mf p$ of a Drinfeld module}\label{sec:ReductionModP}
Let $\phi$ be the Drinfeld module over $F$ specified by $\phi_T := T + \sum_{i=1}^r c_i \tau^i$, with $c_r \neq 0$. Note that there are finitely many primes $\ell$ in  $A$ for which some coefficient of $\phi_T(\tau)$ is not defined modulo $\ell$, or where $\phi_T$ reduces modulo $\ell$ to a polynomial in $\tau$ of degree strictly less than $r$. For all primes $p$ excluding these, we say that  $\phi$ has good reduction at $p$, and we write $\phi \otimes \F_{p}$ for the Drinfeld module over $\F_{p}$ specified by $\phi_T \pmod{p}$. We also write $\phi(\F_{p}) := \left(\phi \otimes \F_{p} \right) (\F_{p})$ when we view $\F_{p}$ as an $A$-module under the action of $\phi \otimes \F_{p}$. Using the fact that $\phi(\F_{p})$ is a finite $A$-module and the classification theorem for finite modules over a principal ideal domain, we write
$$
\phi(\F_{p}) = 
\left( A / d_{1, \mf p}(\phi)A \right)
\times  
\left( A / d_{2, \mf p}(\phi)A \right)
\times \cdots \times  
\left( A / d_{r, \mf p}(\phi)A \right),
$$
where $d_{i, p}(\phi) \in \upone A$ satisfying
$$
d_{1, p}(\phi) \mid
d_{2, p}(\phi) \mid
\cdots \mid
d_{r, p}(\phi).
$$

One computes that the Euler-Poincar\'e characteristic of $\phi(\F_{\mf p})$ is 
$$
\left( d_{1, p}(\phi)A \right)
\left( d_{2, p}(\phi)A \right)
\cdots 
\left( d_{r, p}(\phi)A \right)
;
$$
this is a (principal) ideal in $A$, and we will write $\varpi_p(\phi)$ for the monic generator of this ideal.

%%%%%%%%%%%%%%%%%%%%%%%%%%%%%%%%%%%%%
%%%% BEGIN SECTION %%%%%%%%%%%%%%%%%%
%%%%%%%%%%%%%%%%%%%%%%%%%%%%%%%%%%%%%
\subsection{Torsion, division fields, and Galois representations in generic characteristic}\label{sec:Torsion}
Let $K$ be an $A$-field of generic characteristic. For $n \in A$, define the $n$-torsion of $\phi$, $\phi[n]$, as the set of zeros of $\phi_n$ in $\ol K$. Define the $n$-division field of $\phi$ over $K$, $K(\phi[n])$, as the extension of $K$ obtained by adjoining $\phi[n]$ to $K$. Clearly this extension is Galois over $K$; furthermore, since $\phi[n]$ is $A$-module isomorphic to $(A/nA)^r$, we get a representation 
$$
\ol \rho_{\phi,n}: \Gal(K(\phi[n]) / K) \to GL_2(A/nA)
$$
which we will call the residual representation of $\phi$ at $n$. Note that $\ol \rho_{\phi,n}$ is well-defined up to conjugation (i.e. up to a choice of generators of $(A/nA)^r$). For any prime $\ell$ of $A$, let $\phi[\ell^{\infty}] := \lim_{k \to \infty}\phi[\ell^k]$, and this is $A$-module isomorphic to $(A_{\ell})^r$. Taking an inverse limit on the residual representations at powers of $\ell$ we get the $\ell$-adic representation
$$
\rho_{\phi,\ell}: \Gal(\ol K / K) \to \GL_r(A_{\ell}).
$$
Combining $\ell$-adic representations for all primes $\ell$, we get the absolute Galois representation
$$
\rho_{\phi}: \Gal(\ol K / K) \to \GL_r( \hat A).
$$

More generally, if $\phi$ is a Drinfeld $\mc O$-module over over $K$, where $\mc O$ is an order in a finite imaginary extension of $F$ and $K$ is a finite extension of $F$, we can define $\mf n$-torsion and $\mf n$-division fields of $\phi$ over $K$ exactly as above. Since every ideal $\mc N$ of $\mc O$ is finitely generated, we define $\phi_{\mc N}$ as the (left) greatest common divisor of $\phi_{\mf n}$ where $\mf n$ runs through generators of $\mc N$ in $\mc O$, and $\phi[\mc N]$ as the $\mc O$-module of zeros of $\phi_{\mf n}$ in $\ol K$.

Now, let $\phi$ be a Drinfeld module over $F$ of rank $r$. We say $p$ is a prime of good reduction for $\phi$ if $v_p(c_i) \geq 0$ for $i = 1 , \ldots, r$ and $v_p(c_r) = 0$ where $c_i$ are as in \ref{eqn:PhiAtT}. Let $\mf P_{\phi}$ be the set of primes of good reduction of $\phi$.

\begin{theorem}\label{fact:GekelersFiniteDMThm}
For each prime $p \in \mf P_{\phi}$, there exists a polynomial
$$
\mc P_{\phi,p}(X) := X^r + a_1(\phi,p) X^{r-1} + a_2(\phi,p) X^{r-2} + \cdots + a_r(\phi,p) \in A[X],
$$
with the following properties:
\begin{itemize}
\item For all $n \in A$, the characteristic polynomial of $\ol \rho_{\phi,n}(\Frob_p)$ is precisely $\mc P_{\phi,p}(X) \pmod n$.
\item Any root $\pi$ of $\mc P_{\phi,p}(X)$ satisfies $|\pi| \leq |p|^{1/r}$ (where we are using the normalized extension of $| \cdot |$ to the algebraic closure of $F$).
\item The Euler-Poincar\'e characteristic of $\phi$ at $p$ is $\mc P_{\phi,p}(1)$.
\end{itemize} 
\end{theorem}

Define $a_p := -a_1(\phi,p)$; this is the trace of Frobenius at $p$ for $\phi$. From this discussion it is clear that 
\begin{equation}\label{eqn:HasseBd}
a_p \in A \; \text{and} \; \deg a_p \leq \frac{1}{r} \deg p.
\end{equation}

Also define $\varpi_p :=\varpi_p(\phi) = \mc P_{\phi,p}(1)$ to be the Euler-Poincar\'e characteristic of the reduction of $\phi$ by $p$; it is clear both from the definition and from Theorem \ref{fact:GekelersFiniteDMThm} that we have $\varpi_p \in A$ and $|\varpi_p| = |p|$.

\section{The Chebotarev Density Theorem for Function Fields}
%%%%%%%%%%%%%%%%%%%%%%%%%%%%%%%%%%%%%
%%%% BEGIN SECTION %%%%%%%%%%%%%%%%%%
%%%%%%%%%%%%%%%%%%%%%%%%%%%%%%%%%%%%%
\subsection{An effective version of the Chebotarev Density Theorem}\label{sec:EffectiveChebotarev}

Let $K$ be a finite separable extension of $F$, and let $\F$ be the constant field extension of $K$, i.e. $\F := \ol{\F_q} \cap K$. Let $L$ be a Galois extension of $K$ with Galois group $G := \Gal(L/K)$.
Let $s_L := [\ol{\F}\cap L: \F]$ and $g_K$ the genus of $K$.

Fix a conjugacy-invariant subset $C$ of $G$ satisfying the property that there exists an element $a \in \Z/s_L \Z$ such that, for all $\gs \in C$,
$$
\gs \mid_{\ol{\F}\cap L}
=
(\Frob)^a
$$

(where $\Frob$ is the automorphism of $\ol{\F}\cap L$ given by $\ga \mapsto \ga^{|\F|}$.)

For a prime $\mf p$ of $K$, write $\deg \mf p := \deg(N_{K/F}(\mf p))$. Write
\begin{align*}
&\Pi(x, L/K) := 
\# \{ \mf p \; \text{prime of $K$,} \; \deg \mf p = x: \mf p \; \text{unramified in $L$} \},
\; \text{and} \\
&\Pi_C(x, L/K) := 
\# \{ \mf p \; \text{prime of $K$,} \; \deg \mf p = x: \mf p \; \text{unramified in $L$ and } \Frob_{\mf p} = C\}.
\end{align*}

\begin{theorem}[{\cite[Chapter 5]{Rosen2002}}]\label{fact:PNTforExtensions}
In the notation above, 
$$
\Pi(x, L/K) = \frac{q^x}{x}+ O \left( \frac{q^{x/2}}{x} \right)
$$
\end{theorem}

Define 
$$
|D| := \sum_{\mf p \text{ ramified in }L} \deg \mf p.
$$
Let $\mc D_{L/K}$ be the different of the extension $L/K$. For each ramified prime $\mf p$ in $K$, define $\rho(\mf p)$ such that, for all $\ol{\mf p}$ in $L$ lying above $\mf p$,
$$
v_{\ol{\mf p}} \left( \mc D_{L/K} \right) \leq e(\ol{\mf p} / \mf p) (\rho(\mf p) + 1);
$$
here $e(\ol{\mf p} / \mf p)$ is the ramification index of $\ol{\mf p}$ in $L/K$.
Define 
$$
\rho_{L/K} := \max_{\mf p \text{ ramified in }L}\rho(\mf p).
$$

We now have the notation needed to state an effective version of the Chebotarev Density Theorem for function fields due to V. K. Murty and J. Scherk.

\begin{theorem}[{\cite[Theorem 2]{MurtyScherk1994}}]\label{fact:GeneralChebotarev}
In the notation above, if $x \not\equiv a_c \pmod{s_L}$, then $\Pi_C(x, L/K) = 0$. 
If $x \equiv a_c \pmod{s_L}$, then
\begin{align*}
\left| 
  \Pi_C(x, L/K) - s_L \frac{|C|}{|G|}\Pi(x, L/K)
\right|
\leq & 2|C|^{1/2}
\left(
  \left(
    g_k - 1 + (\rho_{L/K} +1)|D|
  \right)
  \frac{q^{x/2}}{x}
  +
  g_k
  \frac{q^{x/2}}{x}
  +
  \frac{|D|}{2x}
\right)
+ |D|.
\end{align*}
\end{theorem}

%%%%%%%%%%%%%%%%%%%%%%%%%%%%%%%%%%%%%
%%%% BEGIN SECTION %%%%%%%%%%%%%%%%%%
%%%%%%%%%%%%%%%%%%%%%%%%%%%%%%%%%%%%%
\subsection{An effective Chebotarev Density Theorem for Drinfeld modules over $F$}\label{sec:ECForDrinfeldModules}

We use this to make statements about $n$-divison fields of $\phi$ over $F$. For $n \in A$, let $L_n := F(\phi[n])$, $s_n = s_{L_n}$, and $G(n) = \Gal(L_n/F)$.

\begin{theorem}[{\cite[Proposition 6, p. 246]{Gardeyn2002}}]
There exists an ideal $\gD_\phi$ of $A$ such that, uniformly in $n \in A$, we have
$$
\mc D_{L_n / F} \supseteq (nA)^r \gD_{\phi}.
$$
\end{theorem}

Thus, $|D| \leq \deg n + O_\phi(1)$ and for every prime $p$ of $A$, taking a prime $\mf p$ lying above $p$ in $L_n$, we have
$$
v_{\mf p}(\mc D_{L/K}) \leq rv_p(n) + O_\phi(1).
$$
In particular, if $n$ is cube-free, then 
$$
v_{\mf p}(\mc D_{L/K}) = O_{\phi}(1).
$$

Finally, we follows \cite{Zywina2011DrinfeldExample} and note that the map
$$
\Gal(F(\phi[n]) / F)
\to
\Gal(\ol \F_q \cap F(\phi[n]) / \F_q)
$$

is a surjective group homomorphism, and the latter group is cyclic of cardinality $s_n$. Therefore, for any $a \in \Z / s_n \Z$,
$$
\# \left\{
  \gs \in G(n):
  \gs \mid_{\ol \F_q \cap F(\phi[n]) } = (\Frob_q)^a
\right\}
=
\frac{1}{s_n}
|G(n)|
$$

We summarize this discussion in the following statement.

\begin{proposition}\label{fact:DrinfeldEffectiveChebotarev}
Let $\phi$ be a Drinfeld module of rank $r \geq 1$ over $F$, and fix $n \in \upone A$ cube-free. Let $G(n) := \Gal(F(\phi[n]) / F)$ and $s_n := \left[\ol \F_q \cap F(\phi[n]) : \F_q \right]$. Fix a positive integer $x$ and let $C$ be a conjugacy-invariant subset of $G(n)$ such that
$$
\gs \in C \implies \gs \mid_{\ol \F_q \cap F(\phi[n]) } = (\Frob_q)^x
$$

Then,

\begin{align*}
\# &\left\{
  p \in \upone A \; \mathrm{unramified} \; \mathrm{in} \;F(\phi[n]), 
  \deg p = x,
  \left(
  \frac{p}{F(\phi[n]) / F}
  \right)
  \subset C
\right\} \\
=
&
\frac{|C|}{
  \# \left\{
    \gs \in G(n):
    \gs \mid_{\ol \F_q \cap F(\phi[n]) } = (\Frob_q)^a
  \right\}
} 
\frac{q^x}{x}
+
O_{\phi} \left(
  |C|^{1/2}
  \frac{q^{x/2}}{x}
\right).
\end{align*}

\end{proposition}

\section{Rank 2 Drinfeld modules with nontrivial endomorphism ring}\label{sec:Divisibility}

In this section we study rank $2$ Drinfeld modules over $F$ with nontrivial endomorphism ring. We will establish the frequency with which elements $n \in \upone A$ divide $a_p$ and $\varpi_p$ as $p$ runs through primes of good reduction.

\subsection{Applying Pink and R\"{u}tsche's Open Image Theorems}

Let $\phi$ be a rank $2$ Drinfeld module over $F$, with nontrivial endomorphism ring $\End_{\ol F}(\phi)$. As mentioned above, this is an order in an imaginary quadratic extension of $F$. Write $\mc O$ for said order, and write $E$ for said extension of $F$, i.e. the quotient field of $\mc O$. As in the classical case, one can write $E = F \left(\sqrt D \right)$ where $D$ is a fundamental discriminant, and take $\mf f$ to be the unique element of $\upone A$ such that $\mc O = A \left[ \sqrt{\mf f^2 D}\right]$. Recall that $\ell$ is ramified in $E$ if and only if $\ell \mid D$.

Let $E_1$ be the field of definition of $\End_{\ol F}(\phi)$. It is easy to see that this is a finite extension of $F$ containing $E$. We take the tautological lift $\phi$ to the Drinfeld $\mc O$-module $\ot \phi$ 
$$
\ot \phi: \mc O \to E_1 \{ \tau \}.
$$

Note that, as $A$-modules, $\phi[n] = \ot \phi[n]$ for all $n \in A$, and these are isomorphic to (resp.) $(A/nA)^2$ and $(\mc O / n \mc O)^{\textrm{rank} \; \ot \phi}$; it follows that the rank of $\ot \phi$ as a Drinfeld $\mc O$-module is 1, and that there is an isomorphism
\begin{equation} \label{eqn:TorsionIsomorphism}
(A / nA)^2 \cong (\mc O_E / n \mc O_E)
\end{equation}
that is compatible with the action of $\Gal(\ol F / F)$.

We apply \cite[Theorem 0.1]{PinkRutsche2009} to $\ot \phi$, viewing it as a rank $1$ Drinfeld $\mc O$-module over $E_1$. We get that the Galois representation
$$
\rho_{\ot \phi}:
\Gal(\ol E_1 / E_1)
\to
\prod_{\mc L \; \mathrm{prime} \; \mathrm{in} \; E}
(\mc O_{\mc L})^*
$$
has open image. From this we extract the following.

\begin{itemize}
\item[I.] There exists an ideal $\mc M$ of $\mc O$ such that the image of $\rho_{\ot \phi}$ is determined at the level $\mc M$, that is,
$$
\rho_{\ot \phi}
\Gal(\ol E_1 / E_1)
=
\pi^{-1} \left(
\Gal \left( E_1 \left( \ot \phi[\mc M] \right) / E_1 \right)
\right),
$$
\item[] where $\pi$ is the natural projection of $
\prod_{\mc L \; \mathrm{prime} \; \mathrm{in} \; E}
(\mc O_{\mc L})^*
$
 to $(\mc O / \mc M)^*$. Ordering by containment, let $\mc M_{1}$ be the largest such $\mc M$.
\item[II.] Let $\mc D$ be an ideal of $\mc O$. Suppose $\mc D, \mc M_{1}$ are coprime, i.e. $\mc D + \mc M_1 = \mc O$. Then, for any ideal $\mc M$ that contains some power of $\mc M_{1}$,
$$
\Gal \left( E_1 \left( \ot \phi[\mc D \mc M] \right) / E_1 \right)
=
\Gal \left( E_1 \left( \ot \phi[\mc M] \right)/ E_1 \right)
\times
\left( \mc O / \mc D \right)^*.
$$
\item[III.] Let $\mc D_1$ and $\mc D_2$ be ideals of $\mc O$. Suppose they are coprime to each other, and that they are each coprime to $\mc M_{1}$. Then,
$$
\Gal \left( E_1 \left( \ot \phi[\mc D_1 \mc D_2] \right)/ E_1 \right)
=
\left( \mc O / \mc D_1 \right)^*
\times
\left( \mc O / \mc D_2 \right)^*;
$$
\item[] furthermore, $E_1 \left( \ot \phi[\mc D_1] \right)$ and $E_1 \left( \ot \phi[\mc D_2] \right)$ are disjoint over $E_1 $.
\end{itemize}

In particular, if $\ell \nmid N_{E_1/F}\left(\mc M_{1} \right)$, then for any positive integer $k$ the field $F(\phi[\ell^k])$ intersects $E_1 \left( \ot \phi \left( \mc M_{1} \right) \right)$ at a subfield of $E_1$ containing $F$. To control the appearance of $E_1$ and of constant field extensions, we apply a more general statement of the open image theorem of Pink and R\"{u}tsche, also found in \cite{PinkRutsche2009}.

\begin{theorem}[{\cite{PinkRutsche2009}, Theorem 0.2}]
Let $\phi$ be a Drinfeld $A$-module over $F$ of rank $r$. Let $E_1$ be the field of definition of $\End_{\ol F}(\phi)$. Consider the (injective) Galois representation
$$
\rho_{\phi, E_1}:
\Gal\left( \ol E_1 / E_1 \right)
\to
\prod_{\ell}
\Centr_{\GL_r (A_{\ell})}
  (\End_{E_1}(\phi)).
$$
Then, the image of $\rho_{\phi, E_1}$ is open and has finite index in 
$
\prod_{\ell}
\Centr_{\GL_r (A_{\ell})}
  (\End_{E_1}(\phi)).
$
\end{theorem}

Now consider the diagram of continuous homomorphisms
\begin{equation}\label{eqn:OpenImageDiagram}
\begin{gathered}
  \xymatrix{
    \Gal(\ol F / F) 
    \ar[dr]
    \\
    \Gal(\ol E_1 / E_1) 
    \ar@{^{(}->}[u] \ar[r]  
    & 
    \prod_{\ell}
    \Centr_{\GL_r (A_{\ell})}
      (\End_{E_1}(\phi))
  }
\end{gathered}
\end{equation}
where the vertical arrow is inclusion, and the remaining arrows are coming from the Galois representations $\rho_{\phi, F}$ and $\rho_{\phi, E_1}$. Since $\rho_{\phi, E_1}\left(  \Gal(\ol E_1 / E_1)  \right)$ has finite index in 
$
\prod_{\ell}
\Centr_{\GL_r (A_{\ell})}
  (\End_{E_1}(\phi)),
$
so does $\rho_{\phi, F} \left( \Gal(\ol F / F) \right)$. Additionally, $\Gal(\ol F / F)$ is a compact group, so $\rho_{\phi, F} \left( \Gal(\ol F / F) \right)$ must be a closed subgroup. It now follows from the general theory of topological groups that $\rho_{\phi, F} \left( \Gal(\ol F / F) \right)$ is also open. Therefore:
\begin{itemize}
\item[I.] There exists $m \in \upone A$ which completely determines the image of $\rho_{\phi, F}$; let $m_1$ be the smallest such $m$.
\item[II.] Provided $d$ is coprime to $m_1$, $F(\phi[d])$ and $F(\phi[m])$ are disjoint over $F$ whenever $m$ is composed of primes of $m_1$.
\item[III.] If $d_1$ and $d_2$ are coprime to each other and each is coprime to $m_1$, then $F(\phi[d_1])$ and $F(\phi[d_2])$ are disjoint over $F$.
\end{itemize}

Let $K$ be a finite extension of $F$ in $\ol F$. For each prime $\ell$ in $A$, consider the sequence
$$
\left\{
  \left[
    F(\phi[\ell^k]) \cap K : F 
  \right]
\right\}_{k=0}^{\infty}.
$$

By part (III), the sequence is the constant sequence $1,1,\ldots$ for all but finitely many primes $\ell$. For all primes $\ell$, the sequence attains its maximum at some finite value of $k$, since the sequence is supported on positive integers dividing $[K:F]$. For each prime $\ell$, define $k(K,\ell)$ to be the smallest positive integer such that 
$$
\left[
  F(\phi[\ell^{k(K,\ell)}]) \cap K : F 
\right]
=
\left[
  F(\phi[\ell^{k'}]) \cap K : F 
\right]
\; \text{for all} \; k' \geq k(K,\ell),
$$
and define
$$
m(K):= \prod_{\ell} \ell^{k(K,\ell)}
$$
(by the discussion preceding these definitions, this is well-defined). In particular it is clear that, for all $m$ such that $m(K) \mid m$
$$
F(\phi[m]) \cap K = K;
$$
also, for all $d$ coprime to $m(K)$, 
$$
F(\phi[d]) \cap K = F.
$$
In our setting, we take $K$ to include $E_1$ and to include constant field extensions generated by torsion points of $\phi$. The following fact ensures that such a $K$ can be taken to be a finite extension of $F$.

\begin{lemma}[{\cite[Remark 7.1.9, p. 196]{Goss1996}}]\label{fact:GossConstantField}
Let $\mc K$ be the field generated over $F$ by the $n$-torsion points of $\phi$ for all $n \in A$. Then, $\F_{\phi} := \mc K \cap \ol \F_q$ is a finite extension of $\F_q$.
\end{lemma}

Now, take $K :=  \F_{\phi} E_1$, and define
$$
m_{\phi}:= \lcm \left(N_{E_1/F} \left(\mc M_1 \right), \mf f D, m(\F_{\phi} E_1)  \right).
$$

\begin{proposition}\label{fact:ChoseConstantNTER}
Take the notation and setting above. Let $n \in \upone A$ be coprime to $m_{\phi}$. Then, $F(\phi[n])$ is a Galois extension of $F$, disjoint from $E_1$, with constant field $\F_q$, and with Galois group
$$
\Gal(F(\phi[n])/F)) 
= 
\Gal \left( E_1 \left( \ot \phi[d] \right) / E_1 \right)
=
\left( \mc O / n \mc O \right)^*
$$
\end{proposition}

\subsection{Relating matrix conditions to congruence conditions}
Let $n \in \upone A$ be coprime to $m_{\phi}$. By Proposition \ref{fact:ChoseConstantNTER} have the diagram
\begin{equation}\label{eqn:GaloisDiagramNontrivial}
\begin{gathered}
  \xymatrix{
    \Gal(E_1(\ot \phi[n]) / E_1)  
    \ar[d]^{\hat  \rho_{\ot\phi,n}} 
    \ar@{<->}[r] ^{\iota}
    &
    \Gal(F(\phi[n]) / F)
    \ar@{^{(}->}[d]^{\hat \rho_{\phi,n,F}}
    \\
    (\mc O_E / n \mc O_E)^* 
    &
    \GL_2(A / nA)
  }
\end{gathered}
\end{equation}

This diagram allows us to relate calculations of determinants and traces in $\GL_2(A / nA)$ to calculations of norms and traces in $(\mc O / n \mc O)^*$, as we describe below. Observe that there exists an element $w \in \mc O$ with the properties that $\mc O = A[w]$ and the map
\begin{align*}
\mc O &\cong A \times A \\
a+bw &\leftrightarrow (a,b),
\end{align*}
when reduced modulo $n$, is precisely the isomorphism in (\ref{eqn:TorsionIsomorphism}). Let $\gd \in \mc O$. It is known (see for instance \cite[pp.76-79]{CasselsFrohlich1986}) that multiplication by $\gd$ in $\mc O$ corresponds to the linear transformation of $A \times A$ given by a matrix $\gD \in M_2(A)$, and that this matrix satisfies
$$
N_{E/F}(\gd) = \det (\gD) = \gd \ol \gd
\; \text{and}\;
\Tr_{E/F}(\gd) = \tr (\gD) = \gd + \ol \gd;
$$
here $\ol \gd$ is the conjugate of $\gd$ under the action in $\Gal(E/F)$. Reducing modulo $\ell$ and choosing representatives suitably, we get the isomorphism
\begin{align*}
\mc O / \ell \mc O  &\cong A / \ell \times A / \ell \\
a+bw \left(\mod \ell \mc O \right) &\leftrightarrow (a (\mod \ell) ,b (\mod \ell));
\end{align*}

For any $ \gd_{\text{red}} \in \mc O / \ell \mc O$, we find a representative $\gd \in \mc O$, write its associated  matrix $\gD \in M_2(A)$ as described above, and reduce coefficients modulo $\ell$ to obtain $\gD_{\text{red}}$. Now $\gD_{\text{red}}$ satisfies 
\begin{equation} \label{eqn:ReduceToCongruencesNontrivial}
\begin{aligned}
N_{E/F}(\gd) \left(\mod \ell \right) &= \gd_{\text{red}} \ol \gd_{\text{red}}  = \det (\gD_{\text{red}}), \\
\Tr_{E/F}(\gd) \left(\mod \ell \right) &= \gd_{\text{red}} + \ol \gd_{\text{red}} = \tr (\gD_{\text{red}}).
\end{aligned}
\end{equation}

It follows that the diagram in (\ref{eqn:GaloisDiagramNontrivial}) has the following properties.

\begin{lemma} \label{fact:CanUseCongruencesNontrivial}
Take the notation and setting above, with $n \in \upone A$ such that $n$ is coprime to $m_{\phi}$. In \ref{eqn:GaloisDiagramNontrivial}, fix an element in $\Gal(F(\phi[n] / F)$ in the image of $\iota$, let its image in $\GL_2(A / \ell A)$ under $\hat \rho_{\phi,n, F}$ be $\gD_{\mathrm{red}}$, and let its image in $(\mc O / \ell \mc O)^* $ under the composition of the isomorphisms $\hat \rho_{\ot\phi,n}$ and $\iota$ be $\gd_{\mathrm{red}}$. Then, as elements of $\mc O / \ell \mc O$, 
$$
\left(
  \gd_{\mathrm{red}} - 1
\right)
\left(
  \ol \gd_{\mathrm{red}} - 1
\right)  
= 
\det (\gD_{\mathrm{red}}-1)
\; \hspace{6pt} \text{and} \hspace{6pt} \;
 \gd_{\mathrm{red}} + \ol \gd_{\mathrm{red}} = \tr (\gD_{\mathrm{red}}).
$$
\end{lemma}

\subsection{Counting elements in Galois Groups}\label{sec:GaloisGroupsDivisibilityII}

Let $p$ be a prime of $A$ of good reduction for $\phi$. Following \cite{Brown1992}, we say that $p$ is of supersingular reduction for $\phi$ if the endomorphism ring $\End_{\ol \F_{p}}(\phi \otimes \F_{p})$ is a maximal order in a quaternion algebra over $F$, and we say that $p$ is of ordinary reduction otherwise. The following theorem of Brown is an analogue of Deuring's criterion for supersingularity for elliptic curves (see \cite{Deuring1941}), and motivates our choice for the set of primes in consideration. 

\begin{theorem}\label{fact:SupersingularityCriterion} \cite[Lemma 2.9.3, p. 431]{Brown1992}
In the setting above, for any prime $p$ of $A$ of good reduction for $\phi$, the following are equivalent.
\begin{itemize}
\item[a.] $a_p \neq 0$
\item[b.] $p$ is of ordinary reduction for $\phi$
\item[c.] $p$ splits completely in $E$
\item[d.] Writing $\mc P_{\phi,p} = (X - \pi_p)(X- \ol \pi_p) \in \ol F[X]$, we have $F(\pi_p)= E$.
\end{itemize}
\end{theorem}

Therefore, we restrict our attention to those primes $p$ which split completely in $E$; for the primes which are inert in $E$, the study of values of $\varpi_p$ reduces more or less to the study of shifted primes, while the study of values of $a_p$ becomes even less interesting. With this in mind, let $d$ be a cube-free element of $\upone A$ such that $(d, m_{\phi})=1$. We want to estimate the quantity
$$
\# \{
p \in \mf P_{\phi}, \deg p = x: p \; \text{splits completely in } \; E, d \mid a_p
\},
$$
and the analogous quantity for $\varpi_p$. From here on, for $n \in A$ we write $G(n)$ to denote $\Gal(F(\phi[n]))$.

Suppose that $E \neq \F_{q^2}(T)$. Then, the condition of splitting completely in $E$ is independent of the degree of $p$; since the field of constants of $F(\phi[d])$ is $\F_q$, Proposition \ref{fact:DrinfeldEffectiveChebotarev} yields
\begin{align*}
\# &\{
  p \in \mf P_{\phi}, \deg p = x: p \; \text{splits completely in } \; E, d \mid a_p
\} \\
&=
\left\{
  p \in \mf P_{\phi}, \deg p = x:
  \left(
    \frac{p}{E/F}
  \right)
  =1,
  \tr \left(
    \hat \rho_{\phi,d}
    \left(
      \frac{p}{F(\phi[d])/F}
    \right)
  \right)=0
\right\}  \\
&=
\frac{1}{2} \gd^{\mathrm{tr}}(d) \frac{q^x}{x}
+ O_{\phi} \left(
  \left(
    \gd^{\mathrm{tr}}(d) |G(d)|
  \right)^{1/2}
  \frac{q^{x/2}}{x}
  \deg d
\right),
\end{align*}

where 
$$
\gd^{\mathrm{tr}}(d)
=
\frac{
  \# \{
    \gs \in G(d): \tr \left(\hat \rho_{\phi,d} (\gs) \right) = 0
  \}
}
{
|G(d)|
}.
$$

Now suppose $E = \F_{q^2}(T)$. Then, Proposition \ref{fact:DrinfeldEffectiveChebotarev} yields
\begin{equation}\label{eqn:CaseEConstantExtension}
\begin{aligned}
\left\{
  p \in \mf P_{\phi}, \deg p = x:
  \left(
    \frac{p}{E/F}
  \right)
  =1,
  \left(
    \frac{p}{\F_{\phi}(T)/F}
  \right)
  \mid_{\F_{\phi}}
  =
  (\Frob_q)^x,
  \tr \left(
    \hat \rho_{\phi,d}
    \left(
      \frac{p}{F(\phi[d])/F}
    \right)
  \right)=0
\right\} \\
=
\left\{\begin{array}{ll}
0 & \text{if $x$ odd,}\\ 
\gd^{\mathrm{tr}}(d) \frac{q^x}{x}
+ O_{\phi} \left(
  \left(
    \gd^{\mathrm{tr}}(d) |G(d)|
  \right)^{1/2}
  \frac{q^{x/2}}{x}
  \deg d
\right) & \text{if $x$ even.} 
\end{array}\right.
\end{aligned}
\end{equation}

For divisibility of $\varpi_p$ we proceed in the same manner; if $E \neq \F_{q^2}(T)$ we obtain
\begin{align*}
\# &\{
  p \in \mf P_{\phi}, \deg p = x: p \; \text{splits completely in } \; E, d \mid \varpi_p
\} \\
&=
\frac{1}{2} \gd^{\mathrm{ch}}(d) \frac{q^x}{x}
+ O_{\phi} \left(
  \left(
    \gd^{\mathrm{ch}}(d) |G(d)|
  \right)^{1/2}
  \frac{q^{x/2}}{x}
  \deg d
\right),
\end{align*}
where 
$$
\gd^{\mathrm{ch}}(d)
=
\frac{
  \# \{
    \gs \in G(d): \det \left(\hat \rho_{\phi,d} (\gs) - 1 \right) = 0
  \}
}
{
|G(d)|
}.
$$

If $E = \F_{q^2}(T)$, we obtain the analogue of (\ref{eqn:CaseEConstantExtension}), replacing $\gd^{\mathrm{tr}}$ with $\gd^{\mathrm{ch}}$. For the remainder of this section, we assume that $E \neq \F_{q^2}(T)$; the other case will follow from a similar line of reasoning. 

Since $(d, m_{\phi})=1$, the discussion in Section \ref{sec:GaloisGroupsDivisibilityII} implies that 
\begin{equation}\label{eqn:GaloisCountsViaCongruencesInO}
\begin{aligned}
\# & \{
  \gs \in G(d): \tr \left(\hat \rho_{\phi,d} (\gs) \right) = 0
\}
=
\#
\{
  \gs \in (\mc O / d \mc O)^*: \gs + \ol \gs = 0
\} \\
&\text{and}  \\
\# & \{
  \gs \in G(d): \det \left(\hat \rho_{\phi,d} (\gs) - 1 \right) = 0
\}
=
\#
\{
  \gs \in (\mc O / d \mc O)^*: (\gs - 1) (\ol\gs - 1) = 0
\}
\end{aligned}
\end{equation}

We will estimate the size of these sets in the cases where $d$ is a prime.

\begin{proposition}\label{fact:TraceCountByCongruences}
Let $\ell$ be a prime of $A$ that splits completely in $E$ and that does not divide $m_\phi$.
Then, 
\begin{itemize}
\item[I.]$
\# \left\{ 
\gs \in G(\ell): \tr \left(\hat \rho_{\phi,\ell} (\gs) \right) = 0 
\right\}
= |\ell| - 1,
$ and
\item[II.]$
\# \left\{ 
\gs \in G(\ell):  \det \left(\hat \rho_{\phi,\ell} (\gs) - 1 \right) = 0 
\right\}
= 2|\ell| - 3.
$

\end{itemize}
\end{proposition}

\begin{proof}
We proceed by counting solutions to congruences over $\mc O$ modulo $\ell$, using \ref{eqn:GaloisCountsViaCongruencesInO} to relate these to Galois groups. Since $\ell$ splits completely in $E$, we have $\ell \mc O = \mc{L} \ol{\mc L}$ for a prime $\mc L$ in $E$ and its conjugate under the action of $\Gal(E/F)$. Therefore,  $ \mc O / \mc L \cong \mc O / \ol{\mc L} \cong A / \ell A $ and
\begin{equation}\label{eqn:CongruenceCountSplitSetup}
\left(
\mc O / \ell \mc O 
\right)^*
\cong
\left(
\mc O / \mc L \ol{\mc L}
\right)^*
\cong
\left(
\mc O / \mc L 
\right)^*
\left(
\mc O / \ol{\mc L} 
\right)^*.
\end{equation}

To count elements in 
$
\{
  \gs \in (\mc O / \ell \mc O)^*: \gs + \ol \gs = 0
\},
$
we write $\gs \in  (\mc O / \ell \mc O)$ as 
$
\gs = \langle \gs_1, \gs_2 \rangle \in A / \ell A \times A / \ell A 
$
and the involution in $\Gal(E/F)$ becomes
$$
\langle \gs_1, \gs_2 \rangle 
\mapsto
\langle \gs_2, \gs_1 \rangle .
$$
Therefore we are free to choose $\gs_1 \in (A / \ell A)^*$, and $\gs_2 = - \gs_1$. Altogether we have $|\ell|-1$ solutions, proving (I). For (II), we note that
$$
(\gs - 1) (\ol\gs - 1) \equiv 0 \pmod{\mc \ell \ol{\mc L}}
\implies
\gs \equiv 1 (\mod{\mc L})
\; \text{  or  } \;
\gs \equiv 1 (\mod{\ol{\mc L}}),
$$
and by inclusion-exclusion we see that there are $2(|\ell|-1)-1$ solutions, as claimed.
\end{proof}

\begin{proposition}\label{fact:DetCountByCongruences}
Let $\ell$ be a prime of $A$ that is inert in $E$ and that does not divide $m_\phi$.
Then, 
\begin{itemize}
\item[I.]$
\# \left\{ 
\gs \in G(\ell): \tr \left(\hat \rho_{\phi,\ell} (\gs) \right) = 0 
\right\}
= |\ell| - 1,
$ and
\item[II.]$
\# \left\{ 
\gs \in G(\ell):  \det \left(\hat \rho_{\phi,\ell} (\gs) - 1 \right) = 0 
\right\}
= 1.
$
\end{itemize}
\end{proposition}
\begin{proof}
As before,we count solutions to congruences over $\mc O$ modulo $\ell$, using \ref{eqn:GaloisCountsViaCongruencesInO} to relate these to Galois groups. Since $\ell$ is inert in $E$, the ring $\mc O / \ell \mc O$ is a field extension of $\F_{\ell}$; in fact, it can be written as $\F_\ell[\om]$, where $\om$ and $\ol \om$ are the solutions of $X^2 = \mf f^2 D (\mod \ell)$ in the algebraic closure of $\F_{\ell}$. For (I), we parametrize $\mc O / \ell \mc O$ as
$$
\gs  = a + b \om, 
\; \text{ where } \;
a,b \in A / \ell A.
$$
With this parametrization, $\ol{a + b \om} = a + b \ol{\om}$, and
$$
\gs + \ol\gs = 0 \implies 2a=0.
$$
Since $|\ell|$ is odd, we need $a=0$ and $b \neq 0$, giving $|\ell|-1$ solutions to $\gs + \ol\gs = 0$, as claimed. For (II), we use the fact that $(\mc O / \ell \mc O)$ is a field, so 
$$
(\gs - 1)(\ol\gs - 1) = 0
\implies
\gs = 1 
\; \text{ or }
\ol \gs = 1;
$$
which is equivalent to $\gs = 1$.
\end{proof}

\section{Proof of Theorem \ref{thm:MomentsNTER} and Corollaries}

\subsection{Preliminary lemmas}
We keep the notation from the previous section. We begin by proving the following two lemmas.

\begin{lemma}\label{fact:SumDensitiesTraceNTER}
In the notation above,
$$
\sum_{\substack{\deg \ell \leq z \\ l \nmid m_\phi}}
\gd^{\mathrm{tr}}(\ell)
=
\log z 
+ O(1)
\; \text{  and  } \;
\sum_{\substack{\deg \ell \leq z \\ l \nmid m_\phi}}
\left( \gd^{\mathrm{tr}}(\ell) \right)^2
=
O(1).
$$
\end{lemma}

\begin{proof}
For all $\ell \nmid m_\phi$, 
$
|G(\ell)| = \left| \left(\mc O / \ell \mc O  \right)^*  \right| = |\ell|^2 + O(|\ell|);
$
therefore, by Proposition \ref{fact:TraceCountByCongruences},
$$
\gd^{\mathrm{tr}} (\ell)
= 
\frac{1}{|\ell|} 
+ O \left(
\frac{1}{|\ell|^2} 
\right).
$$

The claims now follows from the fact that
$$
\sum_{\deg \ell \leq z} \frac{1}{|\ell|} = \log z + O(1).
$$
and
\begin{equation}\label{eqn:BoundByZetaOf2}
\sum_{\deg \ell \leq z} \frac{1}{|\ell|^2} = O(1).
\end{equation}
\end{proof}

\begin{lemma}\label{fact:SumDensitiesDetNTER}
In the notation above,
$$
\sum_{\substack{\deg \ell \leq z \\ l \nmid m_\phi}}
\gd^{\mathrm{ch}}(\ell)
=
\log z 
+ O(1)
\; \text{  and  } \;
\sum_{\substack{\deg \ell \leq z \\ l \nmid m_\phi}}
\left( \gd^{\mathrm{ch}}(\ell) \right)^2
=
O(1)
$$
\end{lemma}
\begin{proof}
The second statement follows from \ref{eqn:BoundByZetaOf2} combined with the fact that
$$
\gd^{\mathrm{ch}} (\ell) = O \left( \frac{1}{|\ell|} \right)
$$
uniformly for all $\ell$. For the first statement, we use Proposition \ref{fact:DetCountByCongruences} to write
$$
\sum_{\substack{\deg \ell = k \\ l \nmid m_\phi}}
\gd^{\mathrm{ch}} (\ell)
=
2 \sum_{\substack{\deg \ell = k \\   \left(\frac{\ell}{E/F}\right) = 1 }}
  \frac{1}{|\ell|}
+ O \left( \sum_{\deg \ell = k} \frac{1}{|\ell|^2} \right),
$$
and we apply Theorem \ref{fact:GeneralChebotarev} to write this as 
$$
\sum_{\substack{\deg \ell = k \\ l \nmid m_\phi}}
\gd^{\mathrm{ch}} (\ell)
=
  \frac{2}{q^k} \frac{q^k}{2k} + O \left(\frac{1}{q^{3k/2}} + \frac{1}{kq^k}  \right).
$$
The first statement now follows by summing over suitable $k$.

\end{proof}

\subsection{Computing the first moment}

In what follows, we work out the details for $\nu(a_p)$; the computations for moments of $\nu(\varpi_p)$ are essentially the same. Define the set of primes
$$
\mc A :=
\left\{
  p \in \mf P_{\phi}, \deg p = x:
  \left(
    \frac{p}{E/F}
  \right)
  =1,
  \left(
    \frac{p}{\F_{\phi}(T)/F}
  \right)
  \mid_{\F_{\phi}}
  =
  (\Frob_q)^x
\right\}
$$
As noted above, if $E = \F_{q^2}(T)$ then
$$
| \mc A | =
\left\{\begin{array}{ll}
0 & \text{if $x$ odd,}\\ 
\frac{q^x}{x}
+ O_{\phi} \left(
  \frac{q^{x/2}}{x}
\right) & \text{if $x$ even;} 
\end{array}\right. 
$$
otherwise $|\mc A| = \frac{1}{2}\frac{q^x}{x}+ O_{\phi} \left(  \frac{q^{x/2}}{x}\right)$. Assuming that we are in a case where $|\mc A| \neq 0$, we consider the sum
$$
\sum_{p \in \mc A} \nu(a_p).
$$

We fix $0 < \eta_1 < 1$ to be determined later, independent of $x$, and for $n \in A$ define
$$
\nu_1(n) := \# \{ \deg \ell \leq \eta_1 x, \ell \nmid m_\phi, \ell \mid n \}.
$$
Since $\deg a_p \leq \frac{1}{2}x$, we observe that
$$
\nu_1(a_p) = \nu(a_p) + O_{\eta_1, \phi}(1),
$$
and consequentially
$$
\sum_{p \in \mc A} \nu(a_p)
=
\sum_{p \in \mc A} \nu_1(a_p)
+ O_{\eta_1, \phi} \left( |\mc A| \right).
$$

We focus on the first sum on the right hand side; changing the order of summation and using Proposition \ref{fact:DrinfeldEffectiveChebotarev}, we have
\begin{equation}\label{eqn:FirstMomentStep2NTER}
\begin{aligned}
\sum_{p \in \mc A} \nu_1(a_p)
&=
\sum_{\substack{\deg \ell \leq \eta_1 x \\ \ell \nmid m_\phi}}
  \# \{ p \in \mc A: \ell \mid a_p \} \\
&=
|\mc A| 
\sum_{\substack{\deg \ell \leq \eta_1 x \\ l \nmid m_\phi}}
\gd^{\mathrm{tr}}(\ell)
+
O_\phi \left(
\frac{q^{x/2}}{x}
\sum_{\substack{\deg \ell \leq \eta_1 x \\ \ell \nmid m_\phi}}
  |\ell|^{1/2} \deg \ell
\right)
+
O_{\eta_1, \phi} \left( |\mc A| \right) \\
&=
|\mc A| \log x 
+
O_{\eta_1, \phi} \left( |\mc A| \right)
+
O_\phi \left(
q^{x \left(\frac{1}{2} + \frac{3 \eta_1}{2} \right)}
\right).
\end{aligned}
\end{equation}

Taking $0 < \eta_1 < \frac{1}{3}$, this is 
$$
\sum_{p \in \mc A} \nu_1(a_p)
=
| \mc A| \log x
+ 
O_{\phi} \left( |\mc A| \right),
$$
and this establishes Theorem \ref{thm:MomentsNTER} for $k=1$.

\subsection{Computing the second moments}
 As before, take $0 < \eta_2 < 1$, to be determined later, and for $n \in A$ define
$$
\nu_2(n) := \#  \{ \deg \ell \leq \eta_2 x, \ell \nmid m_\phi, \ell \mid n \}.
$$
Then, since the number of divisors of $a_p$ not counted by $\nu_2$ is bounded in terms of $\eta_2$, we can write
$$
\nu(a_p)^2 = \nu_2(a_p)^2 + O_{\eta, \phi}\left( \nu(a_p) \right).
$$
It follows that
$$
\sum_{p \in \mc A} \nu(a_p)^2
=\sum_{p \in \mc A} \nu_2(a_p)^2
+ O \left( |\mc A| (\log x) \right).
$$

We now focus on the sum on the right hand side. Expanding the square and changing the order of summation, we have
\begin{equation}\label{eqn:ExpansionHighMomentsNTER}
\begin{aligned}
\sum_{p \in \mc A} \nu_2(a_p)^2
&=
\sum_{\substack{
  (\ell_1, \ell_2) \\ 
  \deg \ell_i \leq \eta_2 x \\ 
  \ell_i \nmid m_\phi
}}
  \# \{
    p \in \mc A: \ell_i \mid a_p \; \forall i
  \}\\
&=
\sum_{\substack{
  (\ell_1, \ell_2) \\ 
  \deg \ell_i \leq \eta_2 x \\ 
  \ell_i \nmid m_\phi \\
  \ell_1 \neq \ell_2  
}}
  \# \{
    p \in \mc A: \ell_1 \ell_2  \mid a_p
  \}
+
\sum_{\substack{
  \deg \ell \leq \eta_2 x \\ 
  \ell_i \nmid m_\phi \\
}}
  \# \{
    p \in \mc A: \ell \mid a_p 
  \}
\end{aligned}
\end{equation}

We bound the second sum on the right hand side of \ref{eqn:ExpansionHighMomentsNTER}, using the arguments in the previous section, by
\begin{equation}\label{eqn:MomentsNTERPermutationStep}
\sum_{\substack{
  (\ell_1, \ell_2) \\ 
  \deg \ell_i \leq \eta_2 x \\ 
  \ell_i \nmid m_\phi \\
  \ell_1 \neq \ell_2  
}}
  \# \{
    p \in \mc A: \ell_1 \ell_2  \mid a_p
  \}
+
\sum_{\substack{
  \deg \ell \leq \eta_2 x \\ 
  \ell_i \nmid m_\phi \\
}}
  \# \{
    p \in \mc A: \ell \mid a_p 
  \}
=
O \left(
|\mc A|  (\log x)
\right)
\end{equation}

For the first sum on the right hand side of (\ref{eqn:ExpansionHighMomentsNTER}), change the order of summation and use Proposition \ref{fact:DrinfeldEffectiveChebotarev} to get

\begin{align*}
\sum_{\substack{
  (\ell_1, \ell_2) \\ 
  \deg \ell_i \leq \eta_2 x \\ 
  \ell_i \nmid m_\phi \\
  \ell_1 \neq \ell_2  
}}
  \# \{
    p \in \mc A: \ell_1 \ell_2  \mid a_p
  \}
&=
|\mc A|
\sum_{\substack{
  (\ell_1, \ell_2) \\ 
  \deg \ell_i \leq \eta_2 x \\ 
  \ell_i \nmid m_\phi \\
  \ell_1 \neq \ell_2  
}}
  \gd^{\mathrm{tr}}(\ell_1 \ell_2)
+
O_{\phi} \left(
  \frac{q^{x/2}}{x}
  \sum_{\substack{
    (\ell_1, \ell_2) \\ 
    \deg \ell_i \leq \eta_2 x \\ 
    \ell_i \nmid m_\phi \\
    \ell_1 \neq \ell_2  
  }}
    |\ell_1 \ell_2|^{1/2}
\right) \\
&=
|\mc A|
\sum_{\substack{
  (\ell_1, \ell_2) \\ 
  \deg \ell_i \leq \eta_2 x \\ 
  \ell_i \nmid m_\phi \\
  \ell_1 \neq \ell_2  
}}
  \gd^{\mathrm{tr}}(\ell_1)
  \gd^{\mathrm{tr}}(\ell_2) 
+ O_{\phi} \left(  q^{x \left(\frac{1}{2} + 3\eta_2 \right)} \right)
\end{align*}

Taking $0 < \eta_2 < \frac{1}{6}$, the $O$-term is negligible. For the main term, rewrite the sum as
$$
\sum_{\substack{
  (\ell_1, \ell_2) \\ 
  \deg \ell_i \leq \eta_2 x \\ 
  \ell_i \nmid m_\phi \\
  \ell_1 \neq \ell_2  
}}
  \gd^{\mathrm{tr}}(\ell_1)
  \gd^{\mathrm{tr}}(\ell_2) 
=
\left(
\sum_{\substack{\deg \ell \leq \eta_2 x \\ \ell \nmid m_\phi}}
  \gd^{\mathrm{tr}}(\ell)
\right)^2
-
\sum_{\substack{
  \deg \ell \leq \eta_2 x \\ 
  \ell \nmid m_\phi \\
}}
  \gd^{\mathrm{tr}}(\ell_1)^2
$$
Using Lemma using Lemma \ref{fact:SumDensitiesTraceNTER}, the second sum can be bounded by $O(1)$, and the first sum becomes $(\log x)^2 + O( \log x)$. Altogether we have
$$
\sum_{p \in \mc A} \nu(a_p)^2
=
|\mc A| (\log x)^2
+ O \left(
  |\mc A| (\log x)
\right),
$$
which proves this case of Theorem \ref{thm:MomentsNTER}.

\subsection{Proof of Corollaries \ref{thm:CentralMomentsNTER} and \ref{thm:NormalOrderNTER}}
We give the proofs for the statements about $a_p$; the exact same lines of reasoning give the statements about $\varpi_p$. Define the set of primes
$$
\mc A(x) :=
\left\{
  p, \deg p = x:
  \left(
    \frac{p}{E/F}
  \right)
  =1,
  \left(
    \frac{p}{\F_{\phi}(T)/F}
  \right)
  \mid_{\F_{\phi}}
  =
  (\Frob_q)^x
\right\}
$$
If we are in a case where $|\mc A| \neq 0$, we proceed by expanding the square to get
$$
\sum_{p \in \mc A(x)}
  \left(
    \nu(a_p) - \log x
  \right)^2
=
\sum_{p \in \mc A(x)}
  \nu(a_p)^2
-
2 \log x
\sum_{p \in \mc A(x)}
  \nu(a_p)
+
(\log x)^2
\sum_{p \in \mc A(x)}
  1.
$$

Using Theorem \ref{thm:MomentsNTER} to evaluate the sums above, we are left with
$$
\sum_{p \in \mc A(x)}
  \left(
    \nu(a_p) - \log x
  \right)^2
=
\left( 1 - 2 + 1 \right) |\mc A| (\log x)^2
+
O \left (
  |\mc A| (\log x)
\right), 
$$
and the main term vanishes, leaving us with
\begin{equation}\label{eqn:ProveCentralMomentStep}
\sum_{p \in \mc A(x)}
  \left(
    \nu(a_p) - \log x
  \right)^2
\ll
|\mc A| (\log x).
\end{equation}

To prove Corollary \ref{thm:CentralMomentsNTER}, we rewrite the summand in the expression stated and use \ref{eqn:ProveCentralMomentStep} to write
\begin{align*}
\sum_{\substack{
 \deg p \leq x \\ 
 p \; \mathrm{good}\;\mathrm{reduction}\;\mathrm{for}\;\phi \\
 p \; \mathrm{splits}\;\mathrm{in}\; E
}}
  \left(
\nu(a_p) - \log x
  \right)^2
&\ll
\sum_{z = 0}^x
  \sum_{p \in \mc A(z)}
    \left(
      \nu(a_p) - \log z
    \right)^2
+
\sum_{z = 0}^x
  \sum_{p \in \mc A(z)}
    \left(
      \log x - \log z
    \right)^2 \\
&\ll
\sum_{z = 0}^x
  | \mc A(z) |
  (\log x)
  +
  \left( \log \frac{x}{z} \right)^2
\end{align*}

Now, we use the fact that $ | \mc A(z) | \ll \frac{q^z}{z}$ and split the sum to write this as
$$
\sum_{z = 0}^x
  | \mc A(z) |
  (\log x)
  +
  \left( \log \frac{x}{z} \right)^2
\ll
\sum_{0 \leq z \leq \frac{1}{2}x}
  \frac{q^z}{z} (\log x)^2
+
\sum_{\frac{1}{2}x < z \leq x}
  \frac{q^z}{z} (\log x)
\ll
\frac{q^{x/2}}{x} (\log x)^2
+
\frac{q^x}{x} (\log x).
$$

To prove Corollary \ref{thm:NormalOrderNTER}, fix $\e > 0$ and define the set 
\begin{align*}
\mc B_k(x) := 
\left\{
  p \in \bigcup_{z \leq x} \mc A_z:
  \left|
    \nu(a_p) - (\log x)
  \right|
  >
  (\log x)^{\frac{1}{2}} (\log \log x)^{\e}  
\right\}.
\end{align*}

Then, we apply Corollary \ref{thm:CentralMomentsNTER} to write

\begin{align*}
|\mc B_k(x)|
&\ll
\frac{1}{(\log x) (\log \log x)^{2\e}}  
\sum_{p \in \mc B_k(x)}
    \left(
\nu(a_p) - (\log x)
  \right)^2 \\
&\ll
\frac{1}{(\log x) (\log \log x)^{2\e}}  
\sum_{\substack{
 \deg p \leq x \\ 
 p \; \mathrm{good}\;\mathrm{reduction}\;\mathrm{for}\;\phi \\
 p \; \mathrm{splits}\;\mathrm{in}\; E
}}
  \left(
\nu(a_p) - \log x
  \right)^{2} \\
&\ll
\frac{1}{(\log x) (\log \log x)^{2\e}} 
\frac{q^x}{x} (\log x)
\end{align*}
and Corollary \ref{thm:NormalOrderNTER} follows.

\section*{Concluding Remarks}
It is expected that the estimates obtained in Section \ref{sec:GaloisGroupsDivisibilityII} can be used to prove limiting distribution statements for $\nu(a_p)$ and $\nu(\varpi_p)$ in the form of Erd\"{o}s-Kac theorems. This and other related questions will be studied in future work by the author.

\section*{Acknowledgements}

The author thanks A. C. Cojocaru for suggesting this problem and for conversations about it, and Nathan Jones for conversations about the problem. The author also thanks an anonymous referee for pointing out a gap in an earlier version of this paper. The Number Theory Seminar at the University of Illinois at Chicago provided a venue to talk about this work and receive feedback on it. This project is part of the author's doctoral thesis at the University of Illinois at Chicago, and was completed while the author was on an Abraham Lincoln Fellowship; the author thanks the Department of Mathematics, Statistics, and Computer Science and the Graduate College at the University of Illinois at Chicago for providing conditions suitable for research.

\bibliographystyle{amsalpha}
\bibliography{drinfeldmodules}
\end{document}